\documentclass[12pt,twoside, reqno]{amsart}
\usepackage{amsfonts}
\usepackage{latexsym}
\usepackage{amssymb}
\usepackage{amsmath}
\usepackage{dsfont}
\usepackage{multicol}
\usepackage[pdftex,svgnames]{xcolor}
\usepackage{url}
\usepackage[utf8]{inputenc}
\usepackage{mathtools}

\usepackage[english]{babel}

\usepackage{xcolor}

\renewcommand{\le}{\leqslant}

\renewcommand{\ge}{\geqslant}
\renewcommand{\geq}{\geqslant}

\newcommand{\eins}{\mathds{1}}

\newcommand{\N}{{\mathbb N}}

\newcommand{\Z}{{\mathbb Z}}
\newcommand{\F}{\mathfrak{F}}
\newcommand{\I}{\mathcal{I}}
\newcommand{\EU}{\mathcal{EU}}
\newcommand{\U}{\mathfrak{U}}

\newcommand{\G}{\mathfrak{G}}
\newcommand{\B}{\mathfrak{B}}

\newcommand{\ifff}{if and only if }

\theoremstyle{plain}
\newtheorem{theorem}{Theorem}[section]
\newtheorem{lemma}[theorem]{Lemma}
\newtheorem{corollary}[theorem]{Corollary}

\theoremstyle{remark}
\newtheorem{remark}[theorem]{Remark}
\theoremstyle{definition}
\newtheorem{definition}[theorem]{Definition}
\newtheorem{problem}[theorem]{Problem}
\numberwithin{equation}{section}

\begin{document}
\title[Conglomerated filters and representations by ultrafilters]{Conglomerated filters, statistical measures, and representations by ultrafilters}

\author{Vladimir Kadets}
\author{Dmytro Seliutin}
\address{School of Mathematics and Informatics V.N. Karazin Kharkiv  National University,  61022 Kharkiv, Ukraine}
\email{v.kateds@karazin.ua}
\email{selyutind1996@gmail.com}
\thanks{ The research of the first author was partially supported  by project PGC2018-093794-B-I00 (MCIU/AEI/FEDER, UE). The second author was supported by a grant from Akhiezer Foundation \url{http://www.ilt.kharkov.ua/bvi/info/akhiezer_fond/akhiezer_fond_main_e.htm}. On its final stage, the research was supported by the National Research Foundation of Ukraine funded by Ukrainian state budget in frames of project 2020.02/0096 ``Operators in infinite-dimensional spaces:  the interplay between geometry, algebra and topology''}

\begin{abstract}
Using a new concept of conglomerated filter we demonstrate in a purely combinatorial way that none of Erd\"{o}s-Ulam filters or summable filters can be generated by a single statistical measure and consequently they cannot be represented as intersections of countable families of ulrafilters. Minimal families of ultrafilters and their intersections are studied and several open questions are discussed.
\end{abstract}

\maketitle

%%%%%%%%%%%%%%%%%%%
%%         INTRODUCTION
%%%%%%%%%%%%%%%%%%%

\section{Introduction}

In 1937, Henri Cartan (1904--2008), one of the founders of the Bourbaki group,  introduced the concepts of filter and ultrafilter \cite{Cartan1, Cartan2}. These concepts  were among the cornerstones of  Bourbaki's exposition of General Topology \cite{Bourbaki}. For non-metrizable spaces, filter convergence is a good substitute for ordinary convergence of sequences, in particular a Hausdorff space $X$ is compact \ifff  every filter in $X$ has a cluster point. We refer to \cite[Section 16.1]{Kad2018} for a brief introduction to filters and compactness.  Filters and ultrafiters (or equivalent concepts of ideals and maximal ideals of subsets) are widely used in Topology, Model Theory, and Functional Analysis.

Let us recall some definitions. A \emph{filter} $\F$ on a set $\Omega \neq \emptyset$  is a non-empty collection of subsets of $\Omega$ satisfying the following axioms:
\begin{enumerate}
\item[(a)] $\emptyset \notin \F$;
\item[(b)] if $A,B \in \F$ then $A \cap B \in \F$;
\item[(c)] for every $A \in \F$ if $B \supset A$ then $B \in \F$.
\end{enumerate}
The natural ordering on the set of filters on $\Omega$ is defined as follows: $\F_1 \succ \F_2$ if $\F_1 \supset \F_2$. Maximal elements in this order are called \emph{ultrafilters}. The existence of ultrafilters requires the Axiom of Choice, so in this paper we work in Zermelo-Fraenkel-Axiom of Choice (ZFC) system of set theory axioms. For an ultrafilter $\U$ on $\Omega$ the following is true: for every subset $A \subset \Omega$ that does not belong to $\U$, the complement $\Omega \setminus A$  belongs to $\U$. Actually, this property characterizes those filters which are ultrafilters.

In this paper we are interested in filters on $\N$. Given a filter $\F$ in $\N$, a sequence of $x_n$, $n \in {\mathbb{N}}$ in a topological space $X$ is said to be $\F$-\emph{convergent} to $x$  if for every neighborhood $U$ of $x$ the set $\{n \in {\mathbb{N}}\colon x_n \in U\}$ belongs to $\F$. In particular, if one takes as $\F$ the filter of those sets whose complements are finite (the \emph{Fr\'echet filter} $\F_{Fr}$), then $\F_{Fr}$-convergence coincides with the ordinary one. A filter $\F$ on $\N$ is said to be \emph{free} if it dominates the Fr\'echet filter or, equivalently, if the intersection of all elements of $\F$ is empty. In this case, every ordinary convergent sequence is automatically $\F$-convergent. For a free ultrafilter $\U$ on $\N$, every sequence $(x_n)$ in a compact space $X$ is $\U$-convergent, which makes $\U$-limits a powerful and widely used tool. In the sequel, when we say ``filter'' or ``ultrafilter'' we assume that they are free even if we don't say this explicitly.

We use expressions ``collection'' or ``family'' in the same meaning as ``set''. In particular, if we say $W = \{\U_k\}_{k=1}^{n}$ is a collection of filters, we mean that all $\U_k$ are different.

A non-negative finitely additive measure $\mu$ defined on the collection $ 2^{\mathbb N} $ of all subsets of $\mathbb N$ is said to be a \emph{statistical measure} if $\mu(\mathbb N) = 1$ and $\mu(\{k\}) = 0$ for all  $k \in \mathbb N$. Evidently, a statistical measure cannot be countably additive. Statistical measures were introduced in \cite{ChengLinLan2008, ChengLinShi2009, BaoCheng2013} and extensively studied in \cite{ChengHuaZhou2016}. The \emph{filter generated by a statistical measure} $\mu$ is the collection $\F_\mu$ of those subsets $A \subset \N$ for which $\mu(A) = 1$. Conversely, an example of statistical measure is the characteristic function $\eins_\U$ of a free ultrafilter $\U$ on $\mathbb N$:  $\eins_\U(A) = 1$ if $A \in \U$, and $\eins_\U(A) = 0$ if $A \in 2^{\mathbb N} \setminus \U$. Consequently, every free ultrafilter on $\N$ is generated by a statistical measure. To give more examples, one can use the following straightforward observation that rephrases \cite[Theorem 4.4]{ChengHuaZhou2016}.

\begin{remark} \label{rem1-count-inters}
Let $\mu_n$, $n \in \N$, be a sequence of statistical measures, $a_n$, $n \in \N$, be a sequence of positive reals with $\sum_{n \in \N} a_n = 1$, then $ \sum_{n \in \N} a_n \mu_n$ is a statistical measure. In particular, for a sequence  $\U_n$, $n \in \N$ of free ultrafilters on $\N$, the filter
$$
\bigcap_{n \in \N} \U_n=\{A\subset \N\colon A\in \U_n\ \forall n\in \N\}
$$
is generated by the statistical measure  $\sum_{n \in \N} a_n \eins_{\U_n}$.
\end{remark}

Let us also remark that if a statistical measure $\mu$ satisfies that $\mu(A)\in \{0,1\}$ for every $A\subset \N$ then, clearly, $\mu=\eins_\U$ for the ultrafilter 
$$\U =\{A\subset \N\colon \mu(A)=1\}.$$

Besides of this, not too much is known. There are nontrivial examples of statistical measures coming from the Hahn-Banach Theorem, the most prominent of them are the invariant means on countable commutative semigroups, in particular, the generalized Banach limit $\mathrm{Lim}$, see \cite[Section 5.5]{Kad2018}, especially Exercises 8-12 of Subsection 5.5.2. For some of them the corresponding filter cannot be represented as a countable intersection of ultrafilters. The corresponding examples can be extracted from results by Fremlin and Talagrand, see references and a short description in Section \ref{seq-problems}.

According to \cite[Theorem 5.2]{ChengHuaZhou2016}, the Fr\'echet filter is not generated by a statistical measure. In \cite{Kadets2016} the same is shown for the filter $\F_{st}$ of all subsets $A \subset \N$ of natural density $1$. 

The filter  $\F_{st}$ generates the famous \emph{statistical convergence} for sequences, which, together with its various generalizations, is a very popular area of research. Say, Zentralblatt Math. shows 469 documents published between 1981 and 2020 that have words ``statistical convergence'' in their titles. The name ``statistical measure'' is motivated by statistical convergence. The people exploring statistical convergence mostly come to this kind of problems from mathematical analysis, measure theory and functional analysis. Our background and  motivation are the same.

What the authors of \cite{ChengHuaZhou2016} and \cite{Kadets2016} did not know at the moment of the corresponding publications, was that statistical measures (without using this name) were considered earlier by other people, whose motivation were foundations of mathematics like axiomatic set theory,  model theory and  descriptive set theory. The both mentioned above examples of filters that are not generated by a statistical measure as well as many others can be deduced using descriptive set theory approach which we briefly explain below.

Let us identify, as usual, the collection  $2^{\N}$ of all subsets of $\N$ with the Cartesian power $\{0, 1\}^{\N}$. Considering on $\{0, 1\}$ the discrete topology, one generates the standard product topology on $2^{\N}$. It is hidden in the simplified proof of a Solovay's theorem in \cite[Theorem 15.5]{ToWa2016},
without using the words ``statistical measure'', that the filter $\F$ generated by a statistical measure, considered as a subset of the topological space $\{0, 1\}^{\N}$ cannot have the Baire property so, in particular, is not a Borel subset of $\{0, 1\}^{\N}$. Since every ``explicitly defined'' filter (like $\F_{Fr}$, $\F_{st}$, or Erd\"{o}s-Ulam filters and summable filters considered below) is a Borel subset, none of them is generated by a statistical measure. In order to attract more attention of ``mathematical analysis people'' to such kind of reasoning, we  go into some details and give more references in the last section of the paper.

In our paper we address similar kind of questions using an elementary purely combinatorial approach. In Section~\ref{section:poorandconglomerated} we present a simple sufficient condition (called conglomeration property) for a filter to not being generated by a statistical measure. Erd\"os-Ulam filters and summable filters are conglomerated filters, which gives an elementary proof that they are not generated by a statistical measure. In particular, this simplifies a lot the demonstration of the main result of \cite{Kadets2016}.  Outside of this, in Section~\ref{section:intersection} we present some reasoning about filters that are intersections of finite or countable families of ultrafilters. We demonstrate that, in contrast to finite intersections, a representation as an intersection of countable family of ultrafilters is not unique, which makes the problem of determining the existence of such a representation more difficult. A minimal representation as an intersection of countable family of ultrafilters, if exists, is unique, but it is unclear if it exists. We conclude the paper with a list of open questions and related remarks in section~\ref{seq-problems}.

Before we pass to the main part, let us recall some more common terminology about filters. For a given filter $\F$ on $\N$ the corresponding \emph{ideal} of $\F$, $\I = \I(\F)$, is the collection of the complements of the elements of $\F$, that is,
$$
\I(\F) =\{\N \setminus A \colon A \in \F\}.
$$
From the definition of filter, it follows that $\I(\F)$ satisfies the properties of ideals of subsets: $\N\notin \I(\F)$, $\I(\F)$ is closed by finite unions, and if $B_1\in \I(\F)$ and $B_2\subset B_1$, then $B_2\in \I(\F)$.
The corresponding \emph{grill} $\G = \G(\F)$ of $\F$ is the collection of those sets that do not belong to $\I(\F)$ or, equivalently, the collection of those sets that intersect all the elements of $\F$:
$$
\G(\F)=2^{\N} \setminus \I(\F)=\left\{B\in 2^\N\colon B\cap A\neq\emptyset \ \forall A\in \F\right\}.
$$
It is immediate that $\F\subset \G(\F)$. Nowadays, grills are more often called ``co-ideals'' and denoted either $\I^+$ or $\F^*$.  Using the name ``grill" we pay respect to Gustave Choquet who introduced this concept axiomatically and proved \cite{Choquet} that every axiomatically defined grill corresponds to some filter.

A couple of examples can be of help.
\begin{enumerate}
  \item[(1)] If $\F_{Fr}$ is the Fr\'{e}chet filter, then $\I(\F_{Fr})$ is the collection of all finite subsets of $\N$ and $\G(\F_{Fr})$ is the collection of all infinite subsets of $\N$.
  \item[(2)] If $\F_\mu$ is the filter generated by a statistical measure $\mu$, then $\I(\F_\mu) =\{ A \subset \N \colon \mu(A) = 0\}$, and $\G(\F_\mu) =\{ A \subset \N \colon \mu(A) > 0\}$.
\end{enumerate}

For $A \in \G(\F)$ the \emph{trace} $\F|_A$ of $\F$ on $A$ is the collection of all sets of the form $A \cap B$, $B \in \F$. This collection of sets is a filter on $A$.

A family $W$ of subsets of a set $\Omega$ is said to be \emph{centered} if the intersection of any finite collection of members of $W$ is not empty. A family $W$ is centered if and only if there is a filter $\F$ on $\Omega$ containing $W$. A non-empty family $\mathcal{D}$ of subsets of $\Omega$ is called a \emph{filter basis} if $\emptyset \notin \mathcal{D}$ and for every $A,B\in \mathcal{D}$ there is $C\in \mathcal{D}$ such that $C\subset A\cap B$. Given a filter basis $\mathcal{D}$, the family $\F$ of all sets $A\subset \Omega$ which contain at least one element of $\mathcal{D}$ as a subset is a filter, which is called the \emph{filter generated by the basis} $\mathcal{D}$.

We write $\overline{n, m}$ to denote the set of integers of the form $\{n, n+1, \ldots, m\}$.  For  a set $E$ we denote  $\# E$  the number of elements in $E$.

%%%%%%%%%%%%%%%%%%%
%%            SECTION 2
%%%%%%%%%%%%%%%%%%%

 \section{Poor filters and conglomerated filters}\label{section:poorandconglomerated}

Two sets $A, B \subset \N$ are said to be \emph{almost disjoint}, if $\# (A\cap B) < \infty$. For a given free filter $\F$ the sets $A, B \subset \N$ are said to be $\F$-\emph{almost disjoint}, if $A\cap B \in \I(\F)$. Remark, that almost disjointness implies $\F$-almost disjointness, as $\F$ contains the Fr\'{e}chet filter.

Here is the first definition for filter we would like to introduce.

\begin{definition}  \label{def-poor}
A free filter $\F$ on $\N$ is called \emph{poor} if every pairwise $\F$-almost disjoint collection $\mathcal A = \{A_\gamma \}_{\gamma  \in \Gamma}  \subset \G(\F)$ of subsets is at most countable.
\end{definition}

The following easy lemma was stated in \cite[Lemma 2.4]{Kadets2016} for almost disjoint sets. The generalization to $\F$-almost disjointness is straightforward, as the proof is copied from \cite[Lemma 2.4]{Kadets2016} almost word-to word.

\begin{lemma} \label{s2-lem-alm-disj}
Let $\mu$ be a statistical measure, then the corresponding filter $\F = \F_\mu$ is poor.
\end{lemma}

\begin{proof}
Let $A_\gamma \subset \N$, $\gamma \in \Gamma$ be a collection of pairwise $\F$-almost disjoint subsets such that $A_\gamma  \in \G(\F)$ for all $\gamma \in \Gamma$ (that is $\mu(A_\gamma) > 0$). Remark that since $\mu(A) = 0$ for every $A \in \I(\F)$, the finite-additivity formula $\mu\left(\bigcup_{k=1}^n D_k \right ) = \sum_{k=1}^n \mu(D_k)$ remains true for every finite collection of  pairwise $\F$-almost disjoint subsets. Now, for every $n \in \N$ denote $\Gamma_n = \{\gamma \in \Gamma \colon \mu(A_\gamma) > \frac1n\}$. Then for every finite subset $E \subset \Gamma_n$ we have the following estimation for the number of elements of $E$:
$$
\# E < n  \sum_{\gamma \in E} \mu(A_\gamma) = n \mu\Bigl(\bigcup_{\gamma \in E}A_\gamma \Bigr)  \le n \mu(\N) = n.
$$
Consequently, $\# \Gamma_n < n $. Since $\Gamma = \bigcup_{n \in \N}\Gamma_n $,   $\Gamma$ is at most countable.
\end{proof}

Now we are ready to formulate the promised ``simple sufficient condition'' that enables to demonstrate in an elementary way that several standard filters are not generated by a statistical measure.

\begin{definition}  \label{def-conglomerated}
A free filter $\F$ on $\N$ is said to be \emph{conglomerated} if there is a disjoint sequence of sets $D_n \in \I(\F)$, $n \in \N$, such that $\bigcup_{n \in M} D_n \in \G(\F)$ for every infinite subset $M \subset \N$.
\end{definition}

\begin{theorem} \label{thm-suffic-cond}
If $\F$ is a conglomerated filter, then $\F$ is not poor and so, in particular, it is not generated by a statistical measure.
\end{theorem}

\begin{proof}
It is well known (see, for example, \cite[Page 77]{Sierp1958}) that $\N$ contains an uncountable family $\Gamma$ of  pairwise almost disjoint infinite subsets (in fact, a family of continuum cardinality). Define for each $\gamma \in \Gamma$
$$
A_\gamma = \bigcup_{n \in \gamma} D_n.
$$
Then the family $\{A_\gamma \}_{\gamma \in \Gamma}$ is uncountable, pairwise $\F$-almost disjoint and $A_\gamma  \in \G(\F)$ for every  $\gamma \in \Gamma$.
\end{proof}

Our next aim is to present some consequences of the previous theorem. The first immediate consequence deals with the Fr\'{e}chet filter.

\begin{corollary} \label{cor1-lem-alm-disj}
The Fr\'echet filter $\F_{Fr}$ is conglomerated so, in particular, it is not generated by a statistical measure.
\end{corollary}

\begin{proof}
Just take $D_n = \{n\}$.
\end{proof}

For a sequence $s=(s_k)$ of non-negative real numbers such that $\sum_{k=1}^{\infty} s_k=\infty$ the \emph{summable ideal} $\I^{s}$ is defined as the collection of those subsets $A\subset{\mathbb{N}}$ that $\sum_{k\in A} s_k < \infty$. The corresponding filter $\F^{s} = \{\N \setminus A \colon A \in \I^{s}\}$ is called \emph{summable filter}. Then $\I(\F^s) = \I^s$, and
$\G(\F^{s}) = \{B \subset \N  \colon \sum_{k\in A} s_k = \infty\}$.

\begin{theorem} \label{thm-summ-filt}
For every sequence $s=(s_k)$ as above, the corresponding summable filter $\F^{s}$ is  conglomerated.
\end{theorem}
%%%%%%%%%%%%%%%%%dima
\begin{proof}
Denote $d_1 = 0$. We know that $\sum_{k \in \N} s_k = \infty$, so there exists $d_2 \in \N$  such that $ \sum_{k=1}^{d_2} s_k \geq 1$. Obviously $\sum_{k=d_2+1}^{\infty} s_k = \infty$, so there is $d_3 \in \N$, $d_3 > d_2$, with $\sum_{k=d_2+1}^{d_3} s_k \geq 1$. Continuing this procedure, we obtain a sequence of  $d_k$, $d_1 < d_2 < d_3 < \ldots$ such that for all $n \in \N$
$$
 \sum_{k=d_n+1}^{d_{n+1}} s_k \geq 1.
$$
Denote $D_1 =   \overline{d_1 + 1,  d_2}$, $D_2 = \overline{d_2 +  1,  d_3}$, and so on. These $D_n$  form a disjoint sequence of sets. All $D_n$  are finite, so $D_n \in \I(\F^s)$. Finally, for every infinite subset $M \subset \N$ we have
$$
\sum_{k \in \bigcup_{n \in M} D_n} s_k = \sum_{n \in M} \sum_{k=d_n+1}^{d_{n+1}} s_k = \infty,
$$
so $\bigcup_{n \in M} D_n \in \G(F^s)$. This means that  for $\F = \F^{s}$ all the conditions of Definition \ref{def-conglomerated} are fulfilled.
\end{proof}
%%%%%%%%%%%%%%%%%

In the terminology of \cite{Just}, for a sequence $s=(s_k)$ of non-negative real numbers such that $\sum_{k=1}^{\infty} s_k=\infty$, the \emph{Erd\"{o}s-Ulam ideal} $\mathcal{EU}_{s}$ is the ideal of all those $A \subset \N$ that $d_s (A) = 0$  where
\begin{equation*}
d_s(A) =\limsup_{k \to \infty}\frac{\sum_{i\in A\cap \overline{1,k}}s_{i}}{%
\sum_{i=1}^{k}s_{i}}.
\end{equation*}
In order to ensure that $\mathcal{EU}_{s}$ is not the same as the ideal of finite subsets of $\N$ one may, following \cite{Just}, add the condition
\begin{equation*}
\lim_{k \to \infty}\frac{s_{k}}{\sum_{i=1}^{k}s_{i}} = 0,
\end{equation*}
but, for our purposes, this additional restriction is superfluous.
The corresponding filter $\EU^{s} = \{\N \setminus A \colon A \in \mathcal{EU}_{s}\}$ is called the \emph{Erd\"{o}s-Ulam filter}. Then, $\I(\EU^{s}) = \EU_{s}$, and $\G(\EU^{s}) = \{B \subset \N  \colon d_s(B) > 0\}$.

\begin{theorem} \label{thm-eros-ulam}
For every sequence $s=(s_k)$ as above, the corresponding Erd\"{o}s-Ulam filter $\EU^{s}$ is conglomerated.
\end{theorem}
%%%%%%%%%%%%%dima
\begin{proof}
Denote $d_1 = 0$, $d_2 = 1$ and $D_1 = \{1\}$. Then, evidently
$$
\frac{\sum_{i\in D_1}s_{i}}{\sum_{i=1}^{1}s_{i}} = 1 > \frac{1}{2}.
$$
Let us demonstrate the possibility to construct recurrently a sequence of  $d_n$, $d_1 < d_2 < d_3 < \ldots$ and of corresponding $D_n = \overline{d_n +  1,  d_{n+1}}$ in such a way that for all $n \in \N$
\begin{equation} \label{eq-EU-1}
\frac{\sum_{i\in D_n}s_{i}}{\sum_{i=1}^{d_{n+1}}s_{i}}  > \frac{1}{2}.
\end{equation}
Indeed, let $d_j$  be already constructed for $j = 1, \ldots, n$. Since
$$
\lim_{k \to \infty}\frac{\sum_{i = d_n}^{{k}}s_{i}}{\sum_{i=1}^{k}s_{i}} = 1,
$$
there is a $k > d_n$ such that
$$
\frac{\sum_{i = d_n+1}^{{k}}s_{i}}{\sum_{i=1}^{k}s_{i}} > \frac{1}{2}.
$$
It remains to take this particular $k$ as $d_{n+1}$.

Now, when we have all the $d_n$ and corresponding $D_n$, we see, like in the previous theorem, that $D_n$  form a disjoint sequence of sets and, being finite, they are elements of the ideal $\mathcal{EU}_{s}$. Finally, for every infinite subset $M = \{m_1, m_2, \ldots\} \subset \N$ we have that
\begin{align*}
d_s\left(\bigcup_{m \in M} D_m\right) &=\limsup_{k \to \infty}\frac{\sum\limits_{i\in \bigcup_{m \in M} D_m \cap \overline{1,k}}s_{i}}{\sum_{i=1}^{k}s_{i}} \\
&\ge \limsup_{k \to \infty}\frac{\sum\limits_{i\in \bigcup_{m \in M} D_m \cap \overline{1, d_{m_k+1}}}s_{i}}{\sum_{i=1}^{d_{m_k+1}}s_{i}}  \\
& \ge  \limsup_{k \to \infty}\frac{\sum_{i\in  D_{m_{k}}}s_{i}}{\sum_{i=1}^{d_{m_k + 1} }s_{i}} \, \overset{\eqref{eq-EU-1}}\ge \,  \frac{1}{2} \,  >  \, 0.
\end{align*}
So, $\bigcup_{m \in M} D_m \in \G(\EU^{s})$.
\end{proof}
%%%%%%%%%%%%%
The filter $\F_{st}$ generating the famous statistical convergence of sequences is just  $\EU^{s}$ for $s = (1,1,1, \ldots)$. So the previous theorem implies (with a simple and clear proof) the main result of \cite{Kadets2016}, which in turn answered a question from \cite{ChengHuaZhou2016}:

\begin{corollary} \label{cor-stat-filt}
The filter $\F_{st}$ is not generated by a statistical measure.
\end{corollary}

%%%%%%%%%%%%%%%%%%%
%%            SECTION 3
%%%%%%%%%%%%%%%%%%%

 \section{Intersections of families of ultrafilters} \label{section:intersection}
For a collection $W$ of subsets of $2^\N$ we denote by $\cap W$ the intersection of all members of that collection. That is, 
$$
\cap W=\left\{B\subset \N \colon B\in \mathcal{U}\ \forall \mathcal{U}\in W\right\}
$$

\begin{definition}  \label{repesentation}
A collection $W$ of free ultrafilters is said to be a \emph{representation} of the filter $\F$, if $\cap W = \F$
\end{definition}

Let us start with two easy remarks.

\begin {lemma}  \label{lem-filters-incl}
Let $\F_1$, $\F_2$ be free filters on $\N$ with $\F_1 \subset \F_2$. Then $\G(\F_2) \subset \G(\F_1)$ so, in particular, $\F_2\subset \G(\F_1)$.
\end{lemma}

\begin{proof}
As $\F_1 \subset \F_2$, $\I(\F_1)\subset \I(\F_2)$, so $\G(\F_2)=2^\N\setminus \I(\F_2)$ is contained in $2^\N\setminus\I(\F_1)=\G(\F_1)$. Finally, that $\F_2\subset \G(\F_2)$.
\end{proof}

\begin {lemma}  \label{lem-filters-excl}
Let $\F$  be a free filter on $\N$. Then, $A \in 2^\N \setminus \F$ if and only if $(\N \setminus A) \in \G(\F)$.
\end{lemma}

\begin{proof}
If $A\notin \F$, then $\N\setminus A\notin \I(\F)$, so $(\N\setminus A)\in \G(\F)$. Conversely, if $(\N \setminus A) \in \G(\F)$, then $\N\setminus A\notin \I(\F)$, so $A\notin \F$.
\end{proof}

The following easy remark complements the well-known fact that every filter $\F$ is equal to the intersection of all ultrafilters that contain $\F$.

\begin{theorem} \label{thm-continuum}
Let $\F$  be a free filter on $\N$. Then, there exists a family $W$ of ultrafilters on $\N$ such that $\F = \cap W$ and $W$ is of at most continuum cardinality.
\end{theorem}
 %%%%%%%%%dima
\begin{proof}
By Lemma \ref{lem-filters-excl}, for every $A \in 2^\N \setminus \F$ the family of sets $\{\N \setminus A\}\cup \F$ is centered, so there is a filter that contains $\{\N \setminus A\}\cup \F$ and, consequently, we may select an ultrafilter $\U_A$ such that $(\{\N \setminus A\}\cup \F) \subset \U_A$. In other words, $\F \subset \U_A$, but $A \notin \U_A$. Then
$\F = \bigcap\limits_{A \in (2^\N \setminus \F)} \U_A$, so $W = \{\U_A \colon A \in (2^\N \setminus \F)\}$ provides the required representation of $\F$.
\end{proof}

The next lemma explains better the structure of the intersection of a family of filters.

\begin {lemma}  \label{lem-incl-union}
Let $W = \{\F_\gamma\}_{\gamma \in \Gamma}$  be a collection of free filters on $\N$. Then, given $B_\gamma \in \F_\gamma$ for every $\gamma \in \Gamma$, the set $\bigcup_{\gamma \in \Gamma} B_\gamma$ is an element of $\cap W$.
\end{lemma}

\begin{proof}
For every $j \in \Gamma$ we have that  $\bigcup_{\gamma \in \Gamma} B_\gamma \supset B_j$ and  $B_j \in \F_j$, so by axioms of filter $\bigcup_{\gamma \in \Gamma} B_\gamma \in \F_j$.
\end{proof}

Our next goal is to show that a representation of a given filter as a finite intersection of ultrafilters, if exists, is unique.

\begin {lemma}  \label{lem-incl}
Let $W = \{\U_1, \U_2,\ldots,\U_n\}$ be a finite set of free ultrafilters on $\N$, and $\U$ be a free ultrafilter such that $\cap W \subset \U$. Then $\U \in W$.
\end{lemma}

\begin{proof}
We have to show the existence of such $k \in  \overline{1, n}$ that $\U = \U_k$. Let us assume contrary that  $\U \neq \U_k$ for every $k \in  \overline{1, n}$. This means that for every $k \in  \overline{1, n}$ there exists $B_k \in  \U_k$ such that $B_k \notin \U$. As $\U$ is an ultrafilter, so $(\N \setminus {B_k}) \in \U$ for all $k \in  \overline{1, n}$ and, consequently, $$\N \setminus{\bigcup_{k=1}^{n} B_k} = \bigcap_{k=1}^{n} (\N \setminus {B_k}) \in \U.$$ This means that $\bigcup_{k=1}^{n} B_k \notin \U$. But according to Lemma \ref{lem-incl-union},  $\bigcup_{k=1}^{n} B_k \in \cap W \subset \U$, which leads to contradiction.
\end{proof}
%%%%%%%%%%%%%%%%
\begin{theorem} \label{thm-incl}
Let $W_1$ and $W_2$ be finite collections of ultrafilrers on $\N$,  such that $\cap W_1 = \cap W_2$. Then $W_1 = W_2$.
\end{theorem}
 %%%%%%%%%dima
\begin{proof}
For every $\U \in  W_1$  we have that $\U \supset \cap W_1 = \cap W_2$. By Lemma \ref{lem-incl} this gives that $\U \in W_2$. So  $W_1 \subset W_2$. By the same argument $W_2 \subset W_1$.
\end{proof}

\begin{definition}  \label{def-min coll}
A collection $W$ of free ultrafilters consisting of at least two elements is said to be \emph{minimal}, if for every $\U \in W$
$$
\cap W \neq \cap (W \setminus \{\U\}).
$$
A free filter $\F$ on $\N$ is said to be \emph{min-representable} if either it is an ultrafilter or it possess a minimal representation $W$.
\end{definition}

Theorem \ref{thm-incl} implies that every finite set of ultrafilters is minimal, so the intersection of a finite set of ultrafilters is min-representable.
In the sequel, we are going to study what can be said more about minimal representations. First of all, we show that not every filter is min-representable.

\begin {lemma}  \label{lem-filt-notin-ultfilt}
Let $\F_0$  be a free filter and $\U$ be a free ultrafilter on $\N$ such that $\F_0 \not\subset \U$. Denote $\F = \F_0 \cap \U$. Then, for every $D \in \F$ there are $A \in \U$ and $B \in \F_0$ such that $D = A \sqcup B$. Moreover, the trace $\F|_A$ of $\F$ on $A$ is the same as $\U|_A$.
\end{lemma}

\begin{proof}
Since $\F_0 \not\subset \U$, there is a $K \in (\F_0 \setminus \U)$. Denote $B = K \cap D$. We know that both $K$ and $D$ are elements of $\F_0$, so  $B \in \F_0$ as we need. Now, $K \notin \U$, so by the ultrafilter criterion $(\N \setminus K) \in \U$. Consequently, $D \setminus B = D \setminus K = D \cap (\N \setminus K) \in \U$, which means that $A := D \setminus B$ is what we need.

For every $C \in \U|_A$ we have that $C \sqcup B \in \F_0 \cap \U = \F$, so $C = A \cap (C \sqcup B) \in \F|_A$. This demonstrates that the filter $\F|_A$ on $A$ majorates the ultrafilter $\U|_A$ on $A$, so $\F|_A = \U|_A$.
\end{proof}

\begin{theorem} \label{thm-restriction-ultraf}  If a free filter $\F$ possesses a minimal representation $W$, then, for every $\U \in W$, there is  $A \in \U$ such that the trace $\F|_A$ of $\F$ on $A$ is the same as $\U|_A$.
\end{theorem}

\begin{proof}
Denote $\F_0 = \cap (W \setminus \{\U\})$. By minimality,  $\F_0 \not\subset \U$. Also,  $\F = \F_0 \cap \U$. Then, Lemma \ref{lem-filt-notin-ultfilt} applied for $D = \N$ provides us with $A \in \U$ and $B \in \F_0$ such that $\N = A \sqcup B$ and $\F|_A = \U|_A$.
\end{proof}

The last theorem motivates the following definition.

\begin{definition} \label{def-extr-indec}
A free filter $\F$ on $\N$ is said to be \emph{extremely not min-representable}, if for every $A \in \G(\F)$ the trace $\F|_A$ is not an ultrafilter.
\end{definition}

Remark that for an extremely not min-representable filter $\F$ every representation $W$ of $\F$ consisting of more than one element is ``extremely non-minimal" in the following sense: for every $\U \in W$
$$
\F = \cap (W \setminus \{\U\}).
$$

\begin{theorem} \label{thm-repr-for-minimal}
The Frechet filter $\F_{Fr}$,  all Erd\"{o}s-Ulam filters $\EU^{s}$ and all summable filters $\F^{s}$ are extremely not min-representable.
\end{theorem}

\begin{proof}
We present the demonstration for $\F_{Fr}$. The other two cases are also easy to manage. We have that $A \in \G(\F_{Fr})$ \ifff $A$ is infinite. Then, $\F_{Fr}|_A$ consists of those $B \subset A$ such that $A \setminus B$ is finite. So if we write $A$ as a union $A = B_1 \sqcup B_2$ of two infinite sets, then non of them belongs to $\F_{Fr}|_A$. So, $\F_{Fr}|_A$ is not an ultrafilter on $A$.
\end{proof}

\begin{theorem} \label{thm-uncountable-minimal}
For every cardinality $\alpha$ smaller than the continuum, there exist a free filter with a minimal representation of exactly that cardinality.
\end{theorem}

\begin{proof}
Let  $\Gamma \subset 2^\N$ be a family of cardinality  $\alpha$  consisting of pairwise almost disjoint infinite subsets. For each $A \in \Gamma$, pick an ultrafilter $\U_A$ such that $A \in \U_A$. Let us demonstrate that $W = \{\U_A \colon A \in \Gamma\}$ is a minimal collection of ultrafilters (which, of course, is a representation for $\F := \cap W$). Indeed, for every $B \in \Gamma$ and every  $A \in (\Gamma \setminus \{B\})$, the almost disjointness implies that $(A \setminus B) \in \U_A$. Denote $D = \bigcup_{A \in (\Gamma \setminus \{B\}) } A \setminus B$. By Lemma \ref{lem-incl-union}, $D \in \cap (W \setminus \{\U_B\})$. On the other hand, $D \notin \U_B$, which means that $D \notin \F$. So, we demonstrated that $\F \neq \cap (W \setminus \{\U_B\})$ which completes the proof.
\end{proof}

\begin{theorem} \label{thm-repr-for-minimal}
Let $W = \{\U_k\}_{k=1}^{n}$ be a finite or countable minimal collection of free ultrafilters, where $n \in (\N \cup \{\infty\})$, $n \ge 2$, is the number of elements in $W$, and $\F = \cap W$. Then
\begin{enumerate}
\item there exists a partition of $\N$ into disjoint family of subsets $\{N_k\}_{k=1}^{n}$ such that $N_k \in \U_k$ for all $k$.
\item A set $A \subset \N$ is an element of $\F$ if and only if there is a collection of sets $\{A_k\}_{k=1}^{n}$ such that  $A_k \in \U_k$,  $A_k \subset N_k$, and $A = \bigsqcup_{k=1}^n A_k$.
\end{enumerate}
\end{theorem}

\begin{proof}
In order to ensure (1), we may construct the needed subsets $\{N_k\}_{k=1}^{n}$ recurrently, using on each step Lemma \ref{lem-filt-notin-ultfilt}. Indeed, for each $k < n$ denote $\F_k = \bigcap_{j=k+1}^n U_j$. Since $\N \in \F = (\U_1 \cap \F_1)$, and by minimality $\U_1 \not\supset \F_1$, an application of Lemma \ref{lem-filt-notin-ultfilt} provides us with  $N_1 \in \U_1$ and $B_1 \in \F_1$ such that $\N = N_1 \sqcup B_1$. Now $B_1 \in \F_1 = (\U_2 \cap \F_2)$, and by minimality $\U_2 \not\supset \F_2$, so we obtain $N_2 \in \U_2$ and $B_2 \in \F_2$ such that $B_1 = N_2 \sqcup B_2$. Continuing this process, we either stop on $n$-th step if $n < \infty$, or proceed up to infinity. In any case, we get a disjoint family of subsets $\{N_k\}_{k=1}^{n}$ such that $N_k \in \U_k$ for all $k$. If $\bigsqcup_{k=1}^{n} N_k = \N$, we are done. Otherwise, it remains to substitute $N_1$ by $N_1 \sqcup \left(\N \setminus \bigsqcup_{k=1}^{n} N_k\right)$.

In the item (2), one direction of the statement is just Lemma \ref{lem-incl-union}. For the other direction, taking $A \in \F$ it is sufficient to define the needed  $A_k \in \U_k$ by the formula  $A_k = A \cap N_k$.
\end{proof}

The next corollary complements Lemma \ref{lem-incl-union} in the case of intersection of a finite family of ultrafilters.

\begin{corollary} \label{coroll-repr-for-fin-int}
Let $W = \{\U_1, \U_2,\ldots,\U_n\}$, $n \ge 2$, be a finite collection of free ultrafilters, $\F = \cap W$. Then, there exists a partition of $\N$ into a disjoint family of subsets $\{N_k\}_{k=1}^{n}$ such that $N_k \in \U_k$ for every $k \in  \overline{1, n}$, satisfying that a set $A \subset \N$ is an element of $\F$ if and only if $A = \bigsqcup_{k=1}^n A_k$ for some elements $A_k \in \U_k$ with $A_k \subset N_k$ for every $k \in \overline{1, n}$.
\end{corollary}

\begin{proof}
Every finite collection of ultrafilters is minimal by Theorem \ref{thm-incl}, so Theorem \ref{thm-repr-for-minimal} is applicable.
\end{proof}

The descriptions given in Theorem \ref{thm-repr-for-minimal} for finite $n$ and for $n = \infty$ look very similar. Nevertheless, the infinite case loses some nice properties of the finite case, which is reflected in the following theorem.

\begin{theorem}\label{thm_inf_int}
	Let $W = \{\U_k\}_{k=1}^{\infty}$ be a countable minimal collection of free ultrafilters,  and let $\F = \cap W$. Then there exists a free ultrafilter $\U_0$ such that  $\U_0 \supset \F$ but  $\U_0 \notin W$. In particular, the representation of $\F$ as a countable intersection of ultrafilters is not unique: $\F = \bigcap_{k=1}^{\infty} \U_k$ and at the same time $\F = \bigcap_{k=0}^{\infty} \U_k$.
\end{theorem}

\begin{proof}
	Take the sets $N_k$ from Theorem \ref{thm-repr-for-minimal} and consider the following family $G$ of sets: $G = \{A \subset \N\colon \exists j \in \N \ \forall k > j\ A \cap N_k \in \U_k\}$. Evidently, $G \supset \F$. Let us show that
	\begin{enumerate}
		\item[(i)] the family $G$ is a filter;
		\item[(ii)] $\U_k \not\supset G$ $\forall k \in \N$.
	\end{enumerate}

For the item (i) let us check that $G$ verifies the axioms of filter.
	\begin{itemize}
		\item $\emptyset \notin G$, because $\emptyset \notin \U_k$ for each $k \in \N$;
		\item let $A, B \in G$. We have to show that $A \cap B \in G$. As $A, B \in G$, there exists $j_1 \in \N$ such that $A \cap N_k \in \U_k$ for all $k > j_1$, and there exists $j_2 \in \N$ such that $A \cap N_k \in \U_k$ for all $k > j_2$. Denote $j:=\max\{j_1,j_2\}$. This number $j$ can be easily used to show that $A \cap B \in G$;
		\item Let $A \in G$, $D \subset \N$, $A \subset D$. Let's show that $D \in G$. We know that $A \in G$, which means that exists $j_1 \in \N$ such that $A \cap N_k \in \U_k$ for all $k > j_1$. As $D \supset A \supset A \cap N_k$, $A \cap N_k \in \U_k$, and $\U_k$ is a filter, we obtain that $D \in \U_k$ for all $k > j_1$. That is, $D \in G$. We have shown that $G$ is a filter.
	\end{itemize}
In order to  prove the statement (ii), it is enough to remark that for every $k \in \N$ the corresponding $A_k = \bigcap_{j=k+1}^\infty N_j$ belongs to $G$ but $A_k \notin \U_k$ because it does not intersect the set $N_k \in \U_k$.

Let us take as the needed $\U_0$ an arbitrary ultrafilter that majorizes $G$. Then $\U_0 \supset G \supset \F$, and $\U_k \neq \U_0$
	 for all $k \in \N$. The latter is true because $\U_k \not\supset G$ for any $k \in \N$  but $\U_0 \supset G$.
\end{proof}

Although in the infinite case representations are not unique, the MINIMAL representation, if it exists, has to be unique; we will show this below in Theorem \ref{thm_min-unique}.

\begin{definition} \label{def-inavoidable}
Let $\F$ be a free filter and $\U$ be a free ultrafilter on $\N$. $\U$ is said to be \emph{inavoidable} for $\F$, if every representation $W$ of $\F$ contains $\U$ as an element.
\end{definition}

Lemma \ref{lem-filt-notin-ultfilt}  implies that, if $\U$ is an inavoidable ultrafilter for $\F$, then there is  $A \in \U$ such that the trace $\F|_A$ of $\F$ on $A$ is the same as $\U|_A$. The inverse implication is also true.

\begin {lemma}  \label{lem-inavoidable-inverse}
Let $\F$  be a free filter and $\U$ be a free ultrafilter on $\N$. Assume that there is $A \in \U$ such that the trace $\F|_A$ of $\F$ on $A$ is the same as $\U|_A$. Then $\U$ is inavoidable for $\F$.
\end{lemma}

\begin{proof}
Let $\cap W$ be any representation for $\F$ and let $A$ as in the hypothesis. Then, $\N \setminus A \not\in \F$ (otherwise $\emptyset \in \F|_A$), so there is $\widetilde \U \in W$ such that $\N \setminus A \not\in \widetilde \U$. Since $\widetilde \U$ is an ultrafilter, we obtain that $A \in \widetilde \U$. Then, $\widetilde \U|_A \supset \F|_A = \U|_A$, so $\U|_A$ is a base for both $\U$ and $\widetilde \U$ at the same time, that is $\widetilde \U = \U$.
\end{proof}

\begin{theorem}\label{thm_min-unique}
\emph{(a)} If $W$ is a minimal collection of free ultrafilters and $\F = \cap W$, then each $\U \in W$ is inavoidable  for $\F$.

Consequently, \emph{(b)} $\F$ does not have any other minimal representation outside of $W$.
\end{theorem}

\begin{proof}
Item (a) follows from Theorem \ref{thm-restriction-ultraf} and Lemma \ref{lem-inavoidable-inverse}. The statement (b) evidently follows from (a).
\end{proof}

%%%%%%%%%%%%%%%%%%%
%%            SECTION 4
%%%%%%%%%%%%%%%%%%%

 \section{Remarks and open problems} \label{seq-problems}

By now the theory of filters generated by a single statistical measure is making its first steps. The number of examples is limited, consequently one may build a lot of hypotheses which maybe can be destroyed by a clever example. Nevertheless, we find it natural to share with interested colleagues those, maybe childish, questions that we are not able to answer at this stage.

%\begin{problem} \label{prob1}
%Is it true that every filter on $\N$  generated by a statistical measure can be expressed as the intersection of a countable family of ultrafilters?
%\end{problem}

Lemma \ref{s2-lem-alm-disj} says that every filter generated by a statistical measure is poor. So,

\begin{problem} \label{prob2}
Is it true that every poor free filter is generated by a statistical measure?
\end{problem}

According to Theorem \ref{thm-suffic-cond}, every conglomerated filter is not poor. So,

\begin{problem} \label{prob2+}
Is it true that every free filter that is not poor is conglomerated?
\end{problem}

A formally weaker question can be the following:

\begin{problem} \label{prob2++}
Is is true that every non-conglomerated filter is generated by a statistical measure?
\end{problem}

Remark that the answers may depend on continuum hypothesis, so the problems may also be stated as consistency questions.

We next collect some remarks and problems which we will divide in three subsection depending on whether they are related to category theory, measurability, or shift invariance.

\subsection{Remarks and problems related to category theory}
The analysis of the proof of \cite[Theorem 15.5]{ToWa2016} gives the following theorem: if a free filter $\F$ (or, equivalently, the corresponding ideal $\I$) considered as a subspace of the topological space $2^{\N}$ is meager, then there is a family $\Gamma \subset \G(\F)$ of continuum cardinality consisting of  pairwise almost disjoint infinite subsets. Consequently this $\F$ is not poor and cannot be generated by a statistical measure. This theorem is the main ingredient of the proof of already mentioned fact that a filter generated by a statistical measure as a subspace of $2^{\N}$ cannot have the property of Baire.

Recall, that in the product topology on  $2^{\N}$ the standard base of neighborhoods of a set $A \subset \N$ consists of neighborhoods
$$
U_n(A) = \{B \subset \N \colon B \cap \overline{1, n} =  A \cap \overline{1, n}\}.
$$
The proof of \cite[Theorem 15.5]{ToWa2016} mentioned above proceeds as follows. For a meager ideal $\I$ one takes a sequence of nowhere dense subsets $V_n \subset 2^{\N}$  with $\bigcup_{n=1}^\infty V_n \supset \I$ and constructs recurrently a tree $A_0$, $A_1$, $A_{0,0}$, $A_{0,1}$, $A_{1,0}$, $A_{1,1}$,  $A_{0,0,0}$, etc., of finite subsets of  $\N$ and a sequence $m_1 < m_2 < \ldots$ of naturals with the properties that  $A_{t_1, t_2, \ldots, t_n} \subset \overline{1, m_n}$ for any multi-index $t = (t_1, t_2, \ldots, t_n ) \in \{0,1\}^n$; for extensions  $(t_1, t_2, \ldots, t_n, t_{n+1} ) \in \{0,1\}^{n+1}$ of $t$ the inclusions
$$A_{t_1, t_2, \ldots, t_n} \subset A_{t_1, t_2, \ldots, t_n, t_{n+1}}\quad \textrm{and} \quad A_{t_1, t_2, \ldots, t_n, t_{n+1}} \setminus A_{t_1, t_2, \ldots, t_n}  \subset  \overline{m_n + 1, m_{n + 1}}
$$
take place; and that $U_{m_n}(A_{t_1, t_2, \ldots, t_n}) \bigcap V_n = \emptyset$. The corresponding family $\Gamma \subset \G(\F)$ of  pairwise almost disjoint infinite subsets is made up from infinite branches of this tree: for every sequence $(t_1, t_2, \ldots) \in \{0,1\}^\N$ one takes $\bigcup_{n=1}^\infty A_{t_1, t_2, \ldots, t_n}$ as an element of $\Gamma$.

If this tree could be build with the additional property that
$$\quad A_{t_1, t_2, \ldots, t_{n-1}, 0} \setminus A_{t_1, t_2, \ldots, t_{n-1}}= \emptyset \quad \textrm{and} \quad A_{t_1, t_2, \ldots, t_{n-1}, 1} \setminus A_{t_1, t_2, \ldots, t_{n-1}} = D_n, $$
where $D_n$ do not depend on the choice of $t_k$, then $\F$ would be conglomerated. This leads to the following problem.

\begin{problem} \label{prob02+}
Let a free filter $\F  \subset 2^{\N}$ be meager. Does this imply that $\F$ is conglomerated?
\end{problem}

\subsection{Remarks and results related to measurability}
The natural probabilistic measure $p(\{0\}) = p(\{1\}) = \frac12$ on $\{0, 1\}$ induces the standard product probabilistic measure  $\nu$ on $2^{\N}$. The $\sigma$-algebra $\Sigma$ of $\nu$-measurable subsets of   $2^{\N}$ contains the Borel $\sigma$-algebra $\B$ on $2^{\mathbb{N}}$. Denote $\nu^*$ the corresponding outer measure.

If $\U$ is a free ultrafilter, then, according to Sierpi\'nski \cite{Sierp1945}, see also  \cite[ Lemma 464Ca]{Fremlin},  $\nu^*(\U)=1$. Talagrand \cite{Talagrand}, see also  \cite[ Lemma 464Cb]{Fremlin}, demonstrated that  $\nu^*(\F)=1$ for every filter that is a countable intersection of ultrafilters. As $A \mapsto\mathbb{N}\backslash A$ is a
preserving measure bijection of $2^{\mathbb{N}}$, we have that also for such filters
$\nu^*(\I(\F)) = 1$, so the inclusion $2^{\mathbb{N}} \supset \F \sqcup\I(\F)$ says that a countable intersection of ultrafilters is not $\nu$-measurable.

One may ask whether every $\F$ generated by a statistical measure is not $\nu$-measurable. The answer is negative by surprisingly easy probabilistic argument  \cite[Example 464Jb]{Fremlin}. Namely, the coordinate maps $\phi_n \colon 2^{\N} \to \{0, 1\}$, $\phi_n(A) = 1$ if $n \in A$ and $\phi_n(A) = 0$ if $n \notin A$ form an independent sequence of Bernoulli random variables on the probability space $2^{\mathbb{N}}$. Fix an ultrafilter $\U$ on $\N$ and define statistical measure $\mu_{\U}$ by the formula $\mu_{\U}(A) = \lim_{\U}\frac{1}{n}\sum_{k=1}^n \phi_n(A)$. According to the Strong Law of Large Numbers, $\frac{1}{n}\sum_{k=1}^n \phi_n$ tends to $\frac{1}{2}$ with probability $1$, so
$$
\nu\left(\left\{A \in 2^{\N}\colon \mu_{\U}(A) = \frac{1}{2} \right\}\right) = 1.
$$
Consequently, $\nu\left(\F_{\mu_{\U}}\right) = 0$, and $\F_{\mu_{\U}}$ is $\nu$-measurable. Combining this with the Talagrand's result cited above, we obtain the following corollary.

\begin{corollary} \label{cor-stat-not-cinter}
There is a free filter $\F$ generated by a statistical measure which cannot be represented as a countable intersection of ultrafilters. All filters of the form $\F_{\mu_{\U}}$ are such examples.
\end{corollary}

\subsection{Remarks and problems related to shift invariance}\label{ssec4.3}
Recall that a \emph{generalized Banach limit} is a bounded linear functional ${\rm Lim}$ defined on the space $\ell_\infty$ of all bounded sequences of reals and having the following  properties:
\begin{itemize}
\item[-] if $x =(x_1 ,x_2 ,\ldots,x_n ,\ldots)$ has a limit, then $ {\rm Lim} \,x = \lim_{n \to \infty} x_n$;
\item[-] if all $x_k \ge 0$ then $ {\rm Lim} \,x \ge 0$;
\item[-] if $y =(x_2 ,x_3 ,\ldots,x_{n+1} ,\ldots)$ them $ {\rm Lim} \,x = {\rm Lim} \,y$.
\end{itemize}
The existence of such a functional is usually deduced from the Hahn-Banach Theorem. It is known that ${\rm Lim}$ is not unique. For example, in  \cite[Section 16.1.3, Exercise 11]{Kad2018} it is shown that for every free ultrafilter $\U$ on $\N$ the functional that maps each $x =(x_1 ,x_2 ,\ldots,x_n ,\ldots) \in \ell_\infty$ to the $\U$-limit of its arithmetic means $x_1, \frac{x_2 + x_2}{2}, \frac{x_2 + x_2 + x_3}{3}, \ldots$ is a generalized Banach limit.

To each generalized Banach limit ${\rm Lim}$ corresponds the statistical measure $\mu_{{\rm Lim}}$ that sends each $A \subset \N$ to ${\rm Lim}(\eins_A)$, and the filter $\F$ of those  $A \subset \N$ that  ${\rm Lim}(\eins_A) = 1$. The additional property of these filters is their shift-invariance: for every $A = \{n_1, n_2, \ldots\} \in \F$ the corresponding shift $A + 1 = \{n_1+1, n_2+1, \ldots\}$ also lies in $\F$. Corollary \ref{cor-stat-not-cinter} implies that some of shift-invariant filters cannot be represented as a countable intersection of ultrafilters. On the other hand, given a free ultrafilter $\U$ and an integer $n \in \Z$, we can define the shift $\U + n$ as the filter whose base is $\{(A + n)\cap \N \colon A \in \U\}$. Then, $\bigcap_{n \in \Z}(\U + n)$ is a shift-invariant filter which has the form of the intersection of a countable family of ultrafilters. What is not quite clear for us is whether  $\bigcap_{n \in \Z}(\U + n)$ is generated by a shift-invariant statistical measure. This leads to the following question:

\begin{problem} \label{prob4}
May a free filter $\F$ generated by a shift-invariant statistical measure be  equal to the intersection of some countable family of free ultrafilters?
\end{problem}

Some finer than shift-invariance properties of statistical measures were discussed in  \cite{Douwen}, where the word ``diffuse'' instead of ``statistical'' was used. For filters generated by the corresponding measures, the respective variants of Problem \ref{prob4} make sense as well.

The last questions concern the existence of minimal representations.

\begin{problem} \label{prob6}
Assume $\F$ has a countable representation. Does this imply that $\F$ has a minimal representation?
\end{problem}

\begin{problem} \label{prob7}
Does there exist a countable collection  $\{\U_k\}_{k=1}^{\infty}$ of free ultrafilters such that
$$
\bigcap_{k=n}^\infty \U_k = \bigcap_{k=1}^\infty \U_k
$$
for every $n \in \N$?
\end{problem}

\noindent{\bfseries Acknowledgements:} The first author gratefully thanks Prof.
Miguel Mart\'{\i}n for hospitality and fruitful discussions during the visit of the first author to the University of Granada in January-February 2020.

\end{document}